\documentclass[11pt]{amsart}
\usepackage{amsfonts,amssymb,amsthm}
\usepackage{amsmath,amscd}
\usepackage{pstricks}
\usepackage{pstricks,pst-node}
\usepackage{mathrsfs}
\usepackage[all]{xy}

\DeclareMathAlphabet{\mathpzc}{OT1}{pzc}{m}{it}

\theoremstyle{plain}
\newtheorem{Thm}{Theorem}[section]
\newtheorem{Prop}[Thm]{Proposition}
\newtheorem{Lem}[Thm]{Lemma}
\newtheorem{Coro}[Thm]{Corollary}
\theoremstyle{definition}

\numberwithin{equation}{section}

\newcommand{\etam}{\eta_m}
\newcommand{\etak}{\eta_k}

\newcommand{\Kbfj}{K^{\bfj}}

\def\ttv{{v}}

\newcommand{\afbfUslp}{\bfU(\widehat{\frak{sl}}_n)^+}

\newcommand{\afbfUslNp}{{\bf U}(\widehat{\frak{sl}}_N)^+}

\newcommand{\han}{\subseteq}

\newcommand{\lan}{\langle}
\newcommand{\ran}{\rangle}

\newcommand{\leb}{\left[}

\newcommand{\rib}{\right]}

\def\ggp#1#2{\left[\kern-3.2pt\left[{#1\atop #2}\right]\kern-3.2pt\right]}

\def\fS{{\frak S}}

\def\fb{{\frak b}}

\def\sfz{{\mathsf z}}

\newcommand{\g}{{\mathsf g}}
\newcommand{\sff}{{\mathsf f}}
\newcommand{\sfh}{{\mathsf h}}

 \newcommand{\bfOg}{{\boldsymbol \Og}}

\newcommand{\afmsD}{{\mathscr D}^\vtg}

\newcommand{\affSr}{{\fS_{\vtg,r}}}

\newcommand{\afbfHr}{{\boldsymbol{\mathcal H}_\vtg(r)}}
\newcommand{\afbfHra}{{\boldsymbol{\mathcal H}_\vtg(r')}}
\newcommand{\afbfHrad}{{\boldsymbol{\mathcal H}_\vtg(r'')}}

\def\sY{{\mathcal Y}}
\def\sZ{{\mathcal Z}}

\newcommand{\vtg}{{\!\vartriangle\!}}

\newcommand{\dbfHa}{{\boldsymbol{\mathfrak D}_\vtg}(n)}

\newcommand{\dbfHap}{{\boldsymbol{\mathfrak D}^+_\vtg}(n)}

\newcommand{\tri}{\triangle(n)}

\newcommand{\afsl}{\widehat{\frak{sl}}_n}
\newcommand{\afgl}{\widehat{\frak{gl}}_n}

\newcommand{\afbfSr}{{\boldsymbol{\mathcal S}}_\vtg(n,r)}
\newcommand{\afbfSra}{{\boldsymbol{\mathcal S}}_\vtg(n,r')}
\newcommand{\afbfSrad}{{\boldsymbol{\mathcal S}}_\vtg(n,r'')}
\newcommand{\afbfSrada}{{\boldsymbol{\mathcal S}}_\vtg(n,r'+r'')}
\newcommand{\afbfSNr}{{\boldsymbol{\mathcal S}}_\vtg(N,r)}

\newcommand{\afmbnn}{\mathbb N_\vtg^{n}}
\newcommand{\afmbzn}{\mathbb Z_\vtg^{n}}

\newcommand{\afLa}{\Lambda_\vtg}
\newcommand{\afLanr}{\Lambda_\vtg(n,r)}

\newcommand{\afThn}{\Theta_\vtg(n)}

\newcommand{\afThnp}{\Theta_\vtg^+(n)}

\newcommand{\afThnr}{\Theta_\vtg(n,r)}
\newcommand{\afThnra}{\Theta_\vtg(n,r')}
\newcommand{\afThnrad}{\Theta_\vtg(n,r'')}

\newcommand{\afTh}{\Theta_\vtg}

\newcommand{\dzr}{\dot{\zeta}_r}
\newcommand{\dzra}{\dot{\zeta}_{r'}}
\newcommand{\dzrad}{\dot{\zeta}_{r''}}

\def\leq{\leqslant}\def\geq{\geqslant}
\def\le{\leqslant}\def\ge{\geqslant}

\newcommand{\Th}{\Theta}
\newcommand{\dt}{\delta}

\newcommand{\Dt}{\Delta}

\newcommand{\Og}{\Omega}

\newcommand{\og}{\omega}

\newcommand{\vi}{\varphi}

\newcommand{\up}{v}

 \newcommand{\al}{\alpha}
 \newcommand{\bt}{\beta}
 \newcommand{\h}{\widehat}
 \newcommand{\ti}{\widetilde}

\newcommand{\zr}{\zeta_r}

\newcommand{\sg}{\sigma}

\def\th{\theta}

\newcommand{\bop}{\bigoplus}

\newcommand{\ot}{\otimes}

\newcommand{\Ar}{{A,r}}

\newcommand{\ol}{\overline}
\newcommand{\lra}{\longrightarrow}
\newcommand{\ra}{\rightarrow}
 \newcommand{\la}{{\lambda}}

 \newcommand{\mbn}{\mathbb N}
 \newcommand{\mbq}{\mathbb Q}
 
 \newcommand{\mbz}{\mathbb Z}
 
 \newcommand{\bfi}{{\mathbf{i}}}
  \newcommand{\bfd}{{\mathbf{d}}}
 
 \newcommand{\bfj}{{\mathbf{j}}}

 \newcommand{\bfx}{{\mathbf{x}}}
 \newcommand{\bfy}{{\mathbf{y}}}

\newcommand{\bfU}{{\mathbf{U}}}

\newcommand{\ga}{{\gamma}}

\newcommand{\bfBn}{\mathbf{B}(n)}

\newcommand{\bfBNap}{\mathbf{B}(N)^{\mathrm{ap}}}
\newcommand{\dbfBn}{\dot{\mathbf{B}}(n)}
\newcommand{\bfBr}{\mathbf{B}(n,r)}

\newcommand{\End}{\operatorname{End}}

\newcommand{\diag}{\operatorname{diag}}

\def\ro{\text{\rm ro}}
\def\co{\text{\rm co}}

\def\afsygr{{\fS_{\vtg,r}}}

\def\ttx{{\tt x}}
\def\tty{{\tt y}}

\newcommand{\dbfU}{\dot{{\bfU}}(\h{\frak{sl}}_n)}

\newcommand{\afThnap}{\Theta_\vtg(n)^{\rm ap}}
\newcommand{\afThnpap}{\Theta_\vtg^+(n)^{\mathrm{ap}}}
\newcommand{\afThNpap}{\Theta_\vtg^+(N)^{\mathrm{ap}}}
\newcommand{\afThnrap}{\Theta_\vtg(n,r)^{\rm ap}}

\newcommand{\afThNrap}{\Theta_\vtg(N,r)^{\rm ap}}

\begin{document}
\title{The comultiplication of modified quantum affine $\frak{sl}_n$}

\author{Qiang Fu}
\address{Department of Mathematics, Tongji University, Shanghai, 200092, China.}
\email{q.fu@hotmail.com, q.fu@tongji.edu.cn}

\thanks{$^\dagger$ Supported by the National Natural Science Foundation
of China.}

\begin{abstract}
Let $\dbfU$ be the modified quantum affine $\frak{sl}_n$ and let
$\afbfUslNp$ be the positive part of quantum affine $\frak{sl}_N$.
Let $\dbfBn$ be the canonical basis of $\dbfU$ and let
$\bfBNap$ be the canonical basis of $\afbfUslNp$.
It is proved in \cite{FS} that each structure constant for
the multiplication with respect to $\dbfBn$ coincide with a certain structure constant for
the multiplication with respect to $\bfBNap$ for $n<N$. In this paper we use the theory of  affine quantum Schur algebras to prove that
the structure constants for
the comultiplication with respect to  $\dbfBn$ are determined by the structure constants for the comultiplication with respect to $\bfBNap$ for $n<N$. In particular, the positivity property for the comultiplication of  $\dbfU$ follows from the positivity property for the comultiplication  of  $\afbfUslNp$.
\end{abstract}
 \sloppy \maketitle
\section{Introduction}
In \cite{Gro}, Grojnowski gave a geometric construction of the comultiplication $\Delta$ for quantum group of type $A$. As a consequence,
he obtain that it has positive structure constants with respect to the canonical basis for quantum group of type $A$ (see also \cite{JZ}). The geometric description of $\Delta$ was generalized to the affine case by Lusztig \cite{Lu00}.

Let $\afbfSr$ be the affine quantum Schur algebra over $\mbq(\up)$ (see \cite{GV,Gr99,Lu99}). Let $\bfU(\afsl)$ be the quantum affine $\frak{sl}_n$. The algebra $\bfU(\afsl)$ and $\afbfSr$
are related by an algebra homomorphism $\zeta_r:\bfU(\afsl)\ra\afbfSr$
(see Ginzburg--Vasserot \cite{GV} and Lusztig \cite{Lu99}).
The map $\zeta_r$ can be extended to a surjective algebra homomorphism  from $\dbfHa$ to $\afbfSr$, where $\dbfHa$ is the double Ringel--Hall algebra of affine type $A$ (see \cite[3.8.1]{DDF}).

It is well known that the positive part
$\afbfUslp$ of $\bfU(\afsl)$ has a canonical basis $\bfBn^{\text{ap}}$ with remarkable properties (see Kashiwara \cite{Kas1}, Lusztig \cite{Lu91}).
Let $\dbfU$ be the modified form of $\bfU(\afsl)$. The algebra $\dbfU$ is an associative algebra without unity and
the category of $\bfU(\afsl)$-modules of type $1$ is equivalent to the category of unital $\dbfU$-modules.  The algebra $\dbfU$ has a canonical basis $\dbfBn$ constructed by Lusztig \cite{Lubk}. Let $\bfBr$ be the canonical basis of $\afbfSr$ (see Lusztig \cite{Lu99}). The compatibility of $\dbfBn$ and $\bfBr$ was proved by Schiffmann--Vasserot \cite{SV}.

In \cite{FS}, some good relations among the structure constants for
the multiplication with respect to the canonical bases of the three algebras
$\dbfU$, $\afbfSr$ and $\afbfUslNp$ were established. In this paper
we prove that there are similar relations among the structure constants for
the comultiplication with respect to the canonical bases of these algebras. More precisely, we prove in \ref{prop of canonical basis for affine q-Schur algebras} that the structure constants for the comultiplication with respect to $\bfBr$ are determined by that with respect to $\bfBNap$ for $n<N$.  Using  \ref{prop of canonical basis for affine q-Schur algebras},
we prove in \ref{relation dbfBn bfBNap} that the
structure constants for
the comultiplication with respect to $\dbfBn$ are  determined by that with respect to $\bfBNap$ for $n<N$. In particular the positivity property for the comultiplication with respect to  $\dbfBn$ follows from the positivity property for the
comultiplication with respect to  $\bfBNap$.

{\bf Notation:} For a positive integer $n$, let
$\afThn$ be the set of all matrices
$A=(a_{i,j})_{i,j\in\mbz}$ with $a_{i,j}\in\mbn$ (resp. $a_{i,j}\in\mbz$, $a_{i,j}\geq0$ for all $i\neq j$)  such that
\begin{itemize}
\item[(a)]$a_{i,j}=a_{i+n,j+n}$ for $i,j\in\mbz$; \item[(b)] for
every $i\in\mbz$, both sets $\{j\in\mbz\mid a_{i,j}\not=0\}$ and
$\{j\in\mbz\mid a_{j,i}\not=0\}$ are finite.
\end{itemize}
Let $\afThnp=\{A\in\afThn\mid a_{i,j}=0\text{ for }i\geq j\}.$
For $r\geq 0$,
let $\afThnr=\{A\in\afThn\mid\sg(A)=r\},$
where $\sg(A)=\sum_{1\leq i\leq n,\,
j\in\mbz}a_{i,j}.$

Let $\afmbzn=\{(\la_i)_{i\in\mbz}\mid
\la_i\in\mbz,\,\la_i=\la_{i-n}\ \text{for}\ i\in\mbz\}$ and $\afmbnn=\{(\la_i)_{i\in\mbz}\in \afmbzn\mid \la_i\ge0\text{ for  }i\in\mbz\}.$
$\afmbzn$ has a natural structure of abelian group. For $r\geq 0$  let
$\afLanr=\{\la\in\afmbnn\mid\sg(\la)=r\},$
where $\sg(\la)=\sum_{1\leq i\leq n}\la_i$.

Let $\sZ=\mbz[\up,\up^{-1}]$, where $\up$ is an indeterminate.
For $c\in\mbz$ and $a\in\mbn$ let
$\big[{c\atop a}\big]=\prod_{s=1}^a\frac{v^{c-s+1}-v^{-c+s-1}}{v^s-v^{-s}}$.

\section{Preliminaries}
Let
$I=\mbz/n\mbz=\{1,2,\ldots,n\}$ and let $(c_{i,j})_{i,j\in I}$ be the Cartan matrix of affine type $A$.
Let $\dbfHa$ be the double Ringel--Hall algebra  of
affine type $A$. The algebra $\dbfHa$ is isomorphic to the quantum loop algebra $\bfU(\afgl)$ (see \cite[2.5.3]{DDF}).
By {\cite[2.3.1 and 2.3.5]{DDF}}  we have the following result.

\begin{Prop}\label{presentation dHallAlg}
The algebra $\dbfHa$ is the $\mbq(v)$-algebra generated by
$E_i,\ F_i,\  K_i,\ K_i^{-1},\ \sfz^+_s,\ \sfz^-_s,$ for $i\in I,\
s\in\mbz^+$, and relations:
\begin{itemize}
\item[(QGL1)] $K_{i}K_{j}=K_{j}K_{i},\ K_{i}K_{i}^{-1}=1$;

\item[(QGL2)] $K_{i}E_j=\up^{\dt_{i,j}-\dt_{i,j+1}}E_jK_{i}$,
$K_{i}F_j=\up^{-\dt_{i, j}+\dt_{ i,j+1}} F_jK_i$;

\item[(QGL3)] $E_iF_j-F_jE_i=\delta_{i,j}\frac
{\ti K_{i}-{\ti K_{i}}^{-1}}{\up-\up^{-1}}$, where $\ti K_i=
K_iK_{i+1}^{-1}$ (and $\ti K_n=K_nK_1^{-1}$);

\item[(QGL4)]
$\displaystyle\sum_{a+b=1-c_{i,j}}(-1)^a\leb{1-c_{i,j}\atop a}\rib
E_i^{a}E_jE_i^{b}=0$ for $i\not=j$;

\item[(QGL5)]
$\displaystyle\sum_{a+b=1-c_{i,j}}(-1)^a\leb{1-c_{i,j}\atop a}\rib
F_i^{a}F_jF_i^{b}=0$ for $i\not=j$;

\item[(QGL6)]
$\sfz^+_s$ and $\sfz^-_s$ are central elements in $\dbfHa$.
\end{itemize}
where $i,j\in I$ and $s,t\in \mbz^+$.
It is a Hopf algebra with
comultiplication $\Dt$ defined
by
\begin{equation*}\label{Hopf}
\aligned
&\Delta(E_i)=E_i\otimes\ti K_i+1\otimes
E_i,\quad\Delta(F_i)=F_i\otimes
1+\ti K_i^{-1}\otimes F_i,\\
&\Delta(K^{\pm 1}_i)=K^{\pm 1}_i\otimes K^{\pm 1}_i,\quad
\Delta(\sfz_s^\pm)=\sfz_s^\pm\otimes1+1\otimes
\sfz_s^\pm;\\
\endaligned
\end{equation*}
 where $i\in I$ and $s\in \mbz^+$.
\end{Prop}

The extended affine Hecke algebra $\afbfHr$ is defined to be the $\mbq(\up)$-algebra generated by
$T_i,$ $X_j^{\pm 1}(\text{$1\leq i\leq r-1$, $1\leq j\leq r$}),$
 and relations
$$\aligned
 & (T_i+1)(T_i-\up^2)=0,
\;\;T_iT_{i+1}T_i=T_{i+1}T_iT_{i+1},\;\;T_iT_j=T_jT_i\;(|i-j|>1),\\
 & X_iX_i^{-1}=1=X_i^{-1}X_i,\;\; X_iX_j=X_jX_i,\;\; T_iX_iT_i=\ttv^2 X_{i+1},\;\;  X_jT_i=T_iX_j\;(j\not=i,i+1).
\endaligned$$
Let $\bfOg$ be a vector space over $\mbq(\up)$ with basis $\{\og_i\mid i\in\mathbb Z\}$.
Let $I(n,r)=\{(i_1,\ldots,i_r)\in\mbz^r\mid 1\leq i_k\leq
n,\,\forall k\}.$  For
 $\bfi=(i_1,\ldots,i_r)\in\mbz^r$, write
$\og_\bfi=\og_{i_1}\ot\og_{i_2}\ot\cdots\ot \og_{i_r}\in\bfOg^{\ot r}.$
The tensor space $\bfOg^{\ot r}$
admits a right $\afbfHr$-module structure defined by
\begin{equation*}\label{afH action}
\begin{cases}
\og_{\bf i}\cdot X_t^{-1}
=\og_{(i_1,\cdots, i_{t-1},i_t+n,i_{t+1},\cdots,i_r)},\qquad \text{ for all }\bfi\in \mbz^r;\\
{\og_{\bf i}\cdot T_k=\left\{\begin{array}{ll} \ttv^2\og_{\bf
i},\;\;&\text{if $i_k=i_{k+1}$;}\\
\ttv\og_{(i_1,\cdots,i_{k+1},i_k,\cdots,i_r)},\;\;&\text{if $i_k<i_{k+1}$;}\qquad\text{ for all }\bfi\in I(n,r),\\
\ttv\og_{(i_1,\cdots,i_{k+1},i_k,\cdots,i_r)}+(\ttv^2-1)\og_{\bf i},\;\;&\text{if
$i_{k+1}<i_k$,}
\end{array}\right.}
\end{cases}
\end{equation*}
where $1\leq k\leq r-1$ and $1\le t\le r$ (cf. \cite{VV99}).
The algebra $$\afbfSr:=\End_{\afbfHr}(\bfOg^{\ot r})$$
is called an affine quantum Schur algebra.

For $\la\in\afLanr$, let $\fS_\la:=\fS_{(\la_1,\ldots,\la_n)}$
be the corresponding standard Young subgroup of the symmetric group $\fS_r$.
For $\la\in\afLanr$ let $x_\la=v^{\ell(w_{0,\la})}\og_{\bfi_\la}$, where
$w_{0,\la}$ is the longest element in $\frak S_\la$ and
$$\bfi_\la=(\underbrace{1,\ldots,1}_{\la_1},\underbrace{2,\ldots,2}_{\la_2},
\ldots,\underbrace{n,\ldots,n}_{\la_n})\in I(n,r).$$
Then we have $\bfOg^{\ot r}=\oplus_{\la\in\afLanr}x_{\la}\afbfHr$.

The vector space
$\bfOg$ is a natural $\dbfHa$-module with the action
$E_i\cdot \og_s=\dt_{i+1,\bar s}\og_{s-1}$, $F_i\cdot \og_s=\dt_{i,\bar
s}\og_{s+1}$, $K_i^{\pm 1}\cdot \og_s=\up^{\pm\dt_{i,\bar s}}\og_s$, $\sfz_t^+\cdot\og_s=\og_{s-tn}$, {and }
$\sfz_t^-\cdot\og_s=\og_{s+tn},$ where $\bar s$ is the integer $s$ modulo $n$.
The Hopf algebra structure of $\dbfHa$ induces a $\dbfHa$-module $\bfOg^{\ot r}$. We denote by $\zeta_r:\dbfHa\ra\End(\bfOg^{\ot r})$ the corresponding
representation. By \cite[3.8.1]{DDF} we have $\zeta_r(\dbfHa)=\afbfSr$.

For $r',r''\in\mbn$, there is a natural injective algebra homomorphism
$$\vi_{r',r''}:\End(\bfOg^{\ot r'})\ot\End(\bfOg^{\ot r''})\ra\End(\bfOg^{\ot r'+r''})$$
such that $\vi_{r',r''}(f\ot g)(w_1\ot w_2)
=f(w_1)\ot g(w_2)$ for $f\in\End(\bfOg^{\ot r'})$,
$g\in\End(\bfOg^{\ot r''})$, $w_1\in\bfOg^{\ot r'}$,
$w_2\in\bfOg^{\ot r''}$. By restricting the map $\vi_{r',r''}$ to $\afbfSra\ot\afbfSrad$, we obtain an algebra isomorphism $$\vi_{r',r''}:\afbfSra\ot\afbfSrad\ra\vi_{r',r''}(\afbfSra\ot\afbfSrad).$$
It is clear that we have the following  commutative diagram:
$$\begin{CD}
\dbfHa @>\Delta>>\dbfHa\ot\dbfHa\\
@V\zeta_{r'+r''} VV  @VV\zeta_{r'}\ot\zeta_{r''}V\\
\vi_{r',r''}(\afbfSra\ot\afbfSrad) @<<\vi_{r',r''}<
\afbfSra\ot\afbfSrad.
\end{CD}$$
So $\afbfSrada=\zeta_{r'+r''}(\dbfHa)\han\vi_{r',r''}(\afbfSra\ot\afbfSrad)$.
By restricting $\vi_{r'+r''}^{-1}$ to $\afbfSrada$, we obtain an
algebra homomorphism
$$\Delta_{r',r''}:=\vi_{r'+r''}^{-1}:
\afbfSrada\ra\afbfSra\ot\afbfSrad.$$

\section{The connection between $\bfBr$ and $\bfBNap$}

Let $\dbfHap$ be the $\mbq(\up)$-subalgebra of $\dbfHa$ generated by
$E_i$ and $\sfz_s^+$ for $i\in I$ and $s\in \mbz^+$.
Let $\bfBn:=\{\th_A^+\mid A\in\afThnp\}$ be the canonical basis of
$\dbfHap$ (see \cite{VV99}).
For $A=(a_{i,j})\in\afThnp$ let $\bfd(A)=(\sum_{s\leq i,t\geq i+1}a_{s,t})_{i\in \mbz}\in\afmbnn$.
For $\bfj\in\afmbzn$ let $\ti K^\bfj=\prod_{1\leq i\leq n}(\ti K_i)^{j_i}$.
For $A,B\in\afThnp$ we write
\begin{equation}\label{fABC}
\Delta(\th_A^+)=\sum_{B,C\in\afThnp}
\sff_{A,B,C}\th_B^+\ot \th_C^+\ti{K}^{\bfd(B)},
\end{equation}
where $\sff_{A,B,C}\in\sZ$. Note that if $\sff_{A,B,C}\not=0$ then
$\bfd(A)=\bfd(B)+\bfd(C)$.

A matrix $A=(a_{i,j})\in\afThn$ is said to be aperiodic if for every integer $l\neq0$ there exists $1\leq i\leq n$ such that $a_{i,i+l}=0$. Let $\afThnap$ be the set of all aperiodic matrices in $\afThn$. Let
$\afThnpap=\afThnp\cap\afThnap$.

Let $\bfU(\afsl)$ be the $\mbq(\up)$-subalgebra of $\dbfHa$ generated by the elements $E_i$, $F_i$ and $\ti K_i^{\pm 1}$ for $i\in I$.  Let $\bfU(\afsl)^+$ be the  $\mbq(\up)$-subalgebra of $\bfU(\afsl)$ generated by the elements $E_i$ for $i\in I$. Then $\bfU(\afsl)^+$ is isomorphic to the composition algebra of the cyclic quiver $\tri$ (see Ringel \cite{Ri93}).
Let
\begin{equation*}\label{bfBnap}
\bfBn^{\text{ap}}:=\{\th_A^+\mid A\in\afThnpap\}.
\end{equation*}
Then by \cite{Lu92}, we know that the set $\bfBn^{\text{ap}}$
forms a $\mbq(v)$-basis for $\afbfUslp$ and is called the canonical basis of $\afbfUslp$.
The following result are due to Lusztig (see \cite{Lu91} and {\cite[14.4.13]{Lubk}}).
\begin{Thm}\label{positive affine sln}
For $A,B,C\in\afThnpap$ we have $\sff_{A,B,C}\in\mbn[\up,\up^{-1}]$.
\end{Thm}

Let $\afsygr$ be the group consisting of all permutations
$w:\mbz\ra\mbz$ such that $w(i+r)=w(i)+r$ for $i\in\mbz$.
For $\la,\mu\in\afLanr$, let
$\afmsD_\la=\{d\mid d\in\affSr,\ell(wd)=\ell(w)+\ell(d)\text{ for
$w\in\fS_\la$}\}$ and $\afmsD_{\la,\mu}=\afmsD_{\la}\cap{\afmsD_{\mu}}^{-1}$.
For $\la\in\afLanr$,
$1\leq i\leq n$ and $k\in\mbz$ let
$
R_{i+kn}^{\la}=\{\la_{k,i-1}+1,\la_{k,i-1}+2,\ldots,\la_{k,i-1}+\la_i
=\la_{k,i}\},$
where $\la_{k,i-1}=kr+\sum_{1\leq t\leq i-1}\la_t.$
By \cite[7.4]{VV99} (see also \cite[9.2]{DF10}), there is
a bijective map
\begin{equation}\label{jmath}
{\jmath_\vtg}:\{(\la, d,\mu)\mid
d\in\afmsD_{\la,\mu},\la,\mu\in\afLanr\}\lra\afThnr
 \end{equation}
sending $(\la, d,\mu)$ to the matrix $A=(|R_k^\la\cap dR_l^\mu|)_{k,l\in\mbz}$.

The algebra $\afbfSr$ has a normalized  $\mbq(\up)$-basis $\{[A]\mid A\in\afThnr\}$ (cf. \cite[1.9]{Lu99}).
Let
$$\bfBr:=\{\th_{\Ar}\mid A\in\afThnr\}$$ be the
canonical basis of $\afbfSr$ defined by Lusztig \cite{Lu99}. For
$\la,\mu\in\afLanr$ and $d\in\afmsD_{\la,\mu}$ let
$\th_{\la,\mu}^d=\th_{A,r}$, where $A=\jmath_\vtg(\la,d,\mu)$.
Let $\rho$ be the permutation of $\mbz$ sending $j$ to $j+1$ for all $j\in\mbz$.
Then for $\la\in\afLanr$ and $m\in\mbz$ we have
$\th_{\la,\la}^{\rho^{mr}}=[\jmath_\vtg(\la,\rho^{mr},\la)]$.

Assume $r=r'+r''$ with $r',r''\in\mbn$. For $A,B\in\afThnr$
we write
\begin{equation}\label{gABCr}
\Dt_{r',r''}(\th_{A,r})=\sum_{B\in\afThnra\atop C\in\afThnrad}\g_{A,B,C}^{r',r''}\th_{B,r'}\ot
\th_{C,r''}
\end{equation}
where $\g_{A,B,C}^{r',r''}\in\mbq(v)$.

Let $T_\rho =X_1^{-1}\ti T_1^{-1}\cdots\ti T_{r-1}^{-1}\in\afbfHr$, where
$\ti T_i=v^{-1}T_i$. We are now ready to compute
$\Delta_{r',r''}(\th_{\la,\la}^{\rho^{mr}})$. We need the following lemma.

\begin{Lem}\label{Trho}
For $1\leq k\leq r$, we have $(\ti T_{k-1}\cdots\ti T_2\ti T_1)^k
=X_1X_2\cdots X_k$. In particular, we have
$T_\rho^r=X_1^{-1}X_2^{-1}\cdots X_r^{-1}$.
\end{Lem}
\begin{proof}
We apply induction on $k$. The case $k=1,2$ is trivial. We now assume $k>2$.
For $1\leq s\leq k-1$ let $\ttx_{s,k}=(\ti T_{k-1}\cdots\ti T_1X_1)^{k-s}
(\ti T_{k-s}\cdots\ti T_{k-2}\ti T_{k-1})(\ti T_{k-2}\cdots
\ti T_1X_1)^s$.
For $1\leq s\leq k-2$ we have
$\ttx_{s,k}=(\ti T_{k-1}\cdots\ti T_1X_1)^{k-s-1}
\tty_{s,k}(\ti T_{k-s-2}\cdots
\ti T_1X_1)(\ti T_{k-2}\cdots
\ti T_1X_1)^s $
where $\tty_{s,k}=\ti T_{k-1}\cdots\ti T_{k-s}\ti T_{k-s-1}\ti T_{k-s}
\cdots\ti T_{k-1}$. Since $\tty_{s,k}=\ti T_{k-s-1}\tty_{s-1,k}
\ti T_{k-s-1}$ we have $\tty_{s,k}=\ti T_{k-s-1}\cdots\ti T_{k-2}
\ti T_{k-1}\ti T_{k-2}\cdots\ti T_{k-s-1}$. It follows that
$$\ttx_{s,k} =(\ti T_{k-1}\cdots\ti T_1X_1)^{k-s-1}
(\ti T_{k-s-1}\cdots
\ti T_{k-2}\ti T_{k-1})(\ti T_{k-2}\cdots
\ti T_1X_1)^{s+1}=\ttx_{s+1,k}$$ for $1\leq s\leq k-2$.
Consequently,  by induction we have
$(\ti T_{k-1}\cdots\ti T_2\ti T_1)^k=\ttx_{1,k}=\ttx_{k-1,k}
=(\ti T_{k-1}\cdots\ti T_1X_1)(\ti T_1\ti T_2\cdots\ti T_{k-1})
(\ti T_{k-2}\cdots\ti T_2\ti T_1)^{k-1}=X_k\cdots X_2X_1$.
\end{proof}

Assume $r=r'+r''$ with $r',r''\in\mbn$.
There is an injective algebra homomorphism
$$\kappa_{r',r''}:\afbfHra\ot\afbfHrad\ra\afbfHr$$
such that
$\kappa_{r',r''}(T_i\ot 1)=T_i$ ($1\leq i\leq r'-1$),
$\kappa_{r',r''}(X_j\ot 1)=X_j$ ($1\leq j\leq r'$),
$\kappa_{r',r''}(1\ot T_i)=T_{r'+i}$ ($1\leq i\leq r''-1$),
$\kappa_{r',r''}(1\ot X_j)=X_{r'+j}$ ($1\leq j\leq r''$). We will indentify
$\afbfHra\ot\afbfHrad$ as a subalgebra of $\afbfHr$ via $\kappa_{r',r''}$.

\begin{Lem}\label{coproduct formula}
Assume $r=r'+r''$ with $r',r''\in\mbn$.
For $\la\in\afLanr$ and $m\in\mbz$ we have
$$\Delta_{r',r''}(\th_{\la,\la}^{\rho^{mr}})
=\sum_{\al\in\afLa(n,r'),\bt\in\afLa(n,r'')\atop\la=\al+\bt}
\th_{\al,\al}^{\rho^{mr'}}\ot
\th_{\bt,\bt}^{\rho^{mr''}}.$$
\end{Lem}
\begin{proof}
Note that $\th_{\la,\la}^{\rho^{mr}}(x_\mu h)=\dt_{\la,\mu}
x_\mu hT_{\rho^{mr}}$ for $\mu\in\afLanr$ and $h\in\afbfHr$.
Let $\mathscr Y_{\la,m}=\sum_{\al\in\afLa(n,r'),\bt\in\afLa(n,r'')\atop\la=\al+\bt}
\th_{\al,\al}^{\rho^{mr'}}\ot
\th_{\bt,\bt}^{\rho^{mr''}}$.
Clearly, for $\la\in\afLanr$ we have
$$x_\la\afbfHr=\bop_{\ga\in\afLa(n,r'),\,\dt\in\afLa(n,r'')\atop
\la=\ga+\dt}x_\ga\afbfHra\ot x_\dt\afbfHrad.$$
By \ref{Trho}, for $\ga\in\afLa(n,r')$, $\dt\in\afLa(n,r'')$, $h'\in\afbfHra$, $h''\in\afbfHrad$, we have
\begin{equation*}
\begin{split}
\vi_{r',r''}(\mathscr Y_{\la,m})(x_\ga h'\ot x_\dt h'')
&=\dt_{\la,\ga+\dt}x_\ga h'T_{\rho^{mr'}}\ot
x_\dt h''T_{\rho^{mr''}}\\
&=\dt_{\la,\ga+\dt}x_\ga h'
(X_1^{-1} \cdots X_{r'}^{-1})^m\ot
x_\dt h''(X_1^{-1} \cdots X_{r''}^{-1})^m\\
&=
\dt_{\la,\ga+\dt}(x_\ga h'\ot
x_\dt h'')(X_1^{-1} \cdots X_{r'}^{-1}X_{r'+1}^{-1}\cdots
X_{r'+r''}^{-1})^m\\
&=\dt_{\la,\ga+\dt}(x_\ga h'\ot
x_\dt h'')T_\rho^{mr}=\th_{\la,\la}^{\rho^{mr}}(x_\ga h'\ot x_\dt h'').
\end{split}
\end{equation*}
It follows that
$\vi_{r',r''}(\Delta_{r',r''}(\th_{\la,\la}^{\rho^{mr}}))
=\th_{\la,\la}^{\rho^{mr}}=\vi_{r',r''}(\mathscr Y_{\la,m})$ and hence
$\Delta_{r',r''}(\th_{\la,\la}^{\rho^{mr}})
=\mathscr Y_{\la,m}$.
\end{proof}

For $A\in\afThn$ let
$\ro(A)=\bigl(\sum_{j\in\mbz}a_{i,j}\bigr)_{i\in\mbz}$ and
$\co(A)=\bigl(\sum_{i\in\mbz}a_{i,j}\bigr)_{j\in\mbz}.$
By \cite[7.7(2) and 7.9]{DF14} we have the following result.

\begin{Lem}\label{zr(thA+)}
For $A\in\afThnp$ and $\la\in\afLanr$, we have
$\zr(\th_A^+)[\diag(\la)]=
\th_{A+\diag(\la-\co(A)),r}$ {if $\la-\co(A)\in\afmbnn$}, and
$\zr(\th_A^+)[\diag(\la)]=0$ \text{otherwise.}
\end{Lem}

For $\la,\mu\in\afmbzn$ let
$\lan\la,\mu\ran=\sum_{1\leq i\leq n}\la_i\mu_i-\sum_{1\leq i\leq n}\la_i\mu_{i+1}$
for $\la,\mu\in\afmbzn$.
We now use \ref{coproduct formula} and \ref{zr(thA+)} to prove the following formula.

\begin{Coro}\label{ke3}
Assume $r=r'+r''$ with $r',r''\in\mbn$. Let $A\in\afThnp$ and $\la\in\afLa(n,r)$
with $\la-\co(A)\in\afmbnn$. Then
we have
\begin{equation*}
\begin{split}
&\qquad\Delta_{r',r''}(\th_{A+\diag(\la-\co(A)),r})\\
&=\sum_{B,C\in\afThnp,\,\bfd(A)=\bfd(B)+\bfd(C)\atop
\al\in\afLa(n,r'),\,\bt\in\afLa(n,r''),\,
\al+\bt=\la}
\sff_{A,B,C}\up^{\lan\bfd(B),\bt\ran}\th_{B+\diag(\al-\co(B)),r'}
\ot\th_{C+\diag(\bt-\co(C)),r''},
\end{split}
\end{equation*}
where $\sff_{A,B,C}$ is as given in \eqref{fABC}.
\end{Coro}
\begin{proof}
By  \ref{zr(thA+)}  we have
$\Delta_{r',r''}(\th_{A+\diag(\la-\co(A)),r})=
\Delta_{r',r''}(\zeta_r(\th_A^+))\cdot\Delta_{r',r''}([\diag(\la)])=
((\zeta_{r'}\ot\zeta_{r''})\circ\Delta(\th_A^+))\cdot\Delta_{r',r''}([\diag(\la)])$. Now the assertion follows from \eqref{fABC} and  \ref{coproduct formula}.
\end{proof}

For $m\in\mbz$ there is a map
\begin{equation}\label{etam}
\etam:\afThn\ra\afThn
\end{equation}
defined by sending $A=(a_{i,j})_{i,j\in\mbz}$ to $(a_{i,mn+j})_{i,j\in\mbz}$.
The following lemma can be easily checked (see \cite{FS}).
\begin{Lem}\label{ke}
Let  $m\in\mbz$ and  $A\in\afThnr$ with $\la=\ro(A)$ and $\mu=\co(A)$.

$(1)$ If $a_{i,j}=0$ for $1\leq i\leq n$ and $j\leq mn$,
then $\etak(A)\in\afThnp$ for $k\leq m-1$.

$(2)$ We have $\th_{A,r}\cdot\th_{\mu,\mu}^{\rho^{mr}}=\th_{\etam(A),r}
=\th_{\la,\la}^{\rho^{mr}}
\cdot\th_{A,r}$ for $m\in\mbz$.
\end{Lem}

Assume $N\geq n$. There is a natural injective map
$$\ti{\,\,}:\afThn\lra\Th_\vtg(N),\quad A=(a_{i,j})\longmapsto\ti
A=(\ti a_{i,j}),$$ where $\ti A=(\ti a_{i,j})$
is defined by
\begin{equation*}\label{AtoAtilde}
\ti a_{k,l+mN}=\begin{cases} a_{k,l+mn}, &\text{ if }1\leq k,l\leq n;\\
0, &\text{ if either }n< k\leq N\text{ or }n< l\leq N\end{cases}
\end{equation*}
for  $m\in\mbz$. Note that the map $\ti{\,\,}:\afThn\lra\Th_\vtg(N)$ induces a map from $\afThnp$ to $\afTh^+(N)$.
It is easy to see that there is an injective
algebra homomorphism (not sending 1 to 1)
$$\iota_{n,N}:\afbfSr\lra\afbfSNr,\;\;[A]\longmapsto [\ti A]\;\;\text{for $A\in\afThnr$}$$
(see \cite[\S 4.1]{DDF}).
Let $\afThnrap=\afThnap\cap\afThnr$. One can easily prove the following
resutls (see \cite{FS}).
\begin{Lem}\label{prop iota}
Assume $N>n$. Then for $A\in\afThnr$ we have $\ti A\in\afThNrap$ and
$\iota_{n,N}(\th_{A,r})=\th_{\ti A,r}$.
\end{Lem}

We now give a precise relation between the structure constants
of the comultiplication with respect to the canonical basis $\bfBr$
of $\afbfSr$ and that with respect to the canonical basis
$\bfBNap$ of $\afbfUslNp$.

\begin{Thm}\label{prop of canonical basis for affine q-Schur algebras}
Assume $N\geq n$ and $r=r'+r''$. Let $A\in\afThnr$, $B\in\afTh(n,r')$, $C\in\afTh(n,r'')$.

$(1)$ For $X\in\afTh(N,r')$, $Y\in\afTh(N,r'')$, we have
$$\g_{\ti{\etak(A)},X,Y}^{r',r''}=
\begin{cases}
\g_{A,L,M}^{r',r''}&\text{if
$X=\ti{\etak(L)}$ and $Y=\ti{\etak(M)}$ for some $L\in\afTh(n,r')$
and $M\in\afTh(n,r'')$},\\
0&\text{otherwise}
\end{cases}$$
for $k\in\mbz$, where $\g_{A,L,M}^{r',r''}$ is as given in \eqref{gABCr}.

$(2)$ If $N>n$, then there exist $k_0\in\mbz$ such that for $k\leq k_0$, $\ti{\etak(A)},\ti{\etak(B)},\ti{\etak(C)}\in \afThNpap$ and
$$\g_{A,B,C}^{r',r''}
=\up^{\lan\bfd(\ti{\etak(B)}),\co(\ti{\etak(C)})\ran}
\sff_{\ti{\etak(A)},\ti{\etak(B)},\ti{\etak(C)}}\in\mbn[\up,\up^{-1}],$$ where
$\sff_{\ti{\etak(A)},\ti{\etak(B)},\ti{\etak(C)}}$ is as given in \eqref{fABC}.
\end{Thm}
\begin{proof}
Let $\mu=\co(A)$. Then by \ref{ke}(2) and \ref{coproduct formula} we have
\[
\Delta_{r',r''}(\th_{\etak(A),r}) =
\Delta_{r',r''}(\th_{A,r})\Delta_{r',r''}(\th_{\mu,\mu}^{\rho^{kr}})=
\sum_{X\in\afTh(n,r')\atop Y\in\afTh(n,r'')}g_{A,X,Y}^{r',r''}
\th_{\etak(X),r'}\ot\th_{\etak(Y),r''}.
\]
Clearly we have $\Delta_{r',r''}\circ\iota_{n,N}=(\iota_{n,N}\ot
\iota_{n,N})\circ\Delta_{r',r''}$. Thus by  \ref{prop iota} we have
$$\Delta_{r',r''}(\th_{\ti{\etak(A)},r})
=(\iota_{n,N}\ot
\iota_{n,N})(\Delta_{r',r''}(\th_{{\etak(A)},r}))
=
\sum_{X\in\afTh(n,r')\atop Y\in\afTh(n,r'')}g_{A,X,Y}^{r',r''}
\th_{\ti{\etak(X)},r'}\ot\th_{\ti{\etak(Y)},r''}. $$
The statement (1) follows. The statement (2) follows from (1),
\ref{positive affine sln}, \ref{ke3} and \ref{ke}(1).
\end{proof}

\begin{Coro}\label{relation bfBn bfBNap}
Assume $N> n$. For $A,B,C\in\afThnp$ we have
$$\sff_{A,B,C}=\up^{-\lan\bfd(B),\co(C)\ran+\lan
\bfd(\ti B),\co(\ti C)\ran}\sff_{\ti A,\ti B,\ti C}\in\mbn[\up,\up^{-1}],$$ where $\sff_{A,B,C}$ is as given in \eqref{fABC}.
\end{Coro}
\begin{proof}
There exist $\bfx,\bfy\in\afmbnn$ such that
$\bfx+\co(A)=\bfy+\co(B)+\co(C)$. Let
$r'=\sg(B)+\sg(\bfy)$ and $r''=\sg(C)$.
By \ref{ke3}, \ref{prop of canonical basis for affine q-Schur algebras}(1)
and \ref{positive affine sln} we have
$\sff_{A,B,C}=\up^{-\lan\bfd(B),\co(C)\ran}
\g_{A+\diag(\bfx),B+\diag(\bfy),C}^{r',r''}=
\up^{-\lan\bfd(B),\co(C)\ran}
\g_{\ti A+\ti{\diag(\bfx)},\ti B+\ti{\diag(\bfy)},\ti C}^{r',r''}
=\up^{-\lan\bfd(B),\co(C)\ran+\lan
\bfd(\ti B),\co(\ti C)\ran}\sff_{\ti A,\ti B,\ti C}
\in\mbn[\up,\up^{-1}].$
\end{proof}

\section{The connection between $\dbfBn$ and $\bfBNap$}

Let $X$ be the quotient of $\afmbzn$ by the subgroup generated by
the element ${\bf 1}$, where ${\bf 1}_i=1$ for all $i$.
For $\la\in\afmbzn$ let $\bar\la\in X$ be the image of $\la$ in $X$. Let $Y=\{\mu\in\afmbzn\mid\sum_{1\leq i\leq n}\mu_i=0\}$. For $\bar\la\in X$ and $\mu\in Y$ we set $ \mu\cdot \bar\la=\sum_{1\leq
i\leq n}\la_i\mu_i$.
For $\bar\la,\bar\mu\in X$ we set
${}_{\bar\la}\bfU(\afsl)_{\bar\mu}=\bfU(\afsl)/{}_{\bar\la}
I_{\bar\mu},$
where ${}_{\bar\la}I_{\bar\mu}=
\sum_{\bfj\in Y}(\Kbfj-
 \up^{\bfj\cdot\bar\la})\bfU(\afsl)+\sum_{\bfj\in Y}\bfU(\afsl)(\Kbfj
 -\up^{\bfj\cdot\bar\mu}) $
and $K^{\bfj}=\prod_{1\leq i\leq n}K_i^{j_i}$.
Let  $$\dbfU:=\bop\limits_{\bar\la,\bar\mu\in X}
{}_{\bar \la}\bfU(\afsl)_{\bar\mu}.$$
There is a natural algebra structure on $\dbfU$ inherited from that of $\bfU(\afsl)$ (see \cite[23.1.1]{Lubk}).

Let $\pi_{\bar\la,\bar\mu}:\bfU(\afsl)\ra
{}_{\bar\la}\bfU(\afsl)_{\bar\mu}$ be the canonical projection.
For $\bar\la\in X$ let $1_{\bar\la}=\pi_{\bar\la,\bar\la}(1)$.
The map $\zr$ defined in \S2 induces an
algebra homomorphism
$$\dzr:\dbfU\ra\afbfSr$$ such that
$\dzr(\pi_{\bar\la,\bar\mu}(u))=\dzr(1_{\bar\la})
\zeta_r(u)\dzr(1_{\bar\mu})$ for $u\in\bfU(\afsl)$ and
$\bar\la,\bar\mu\in X$, and
$\dzr(1_{\bar\la})=
[\diag(\al)]$ \text{if $\bar\la=\bar\al$ for some $\al\in\afLanr$,}
$\dzr(1_{\bar\la})=0$ otherwise.

For $\bar\al,\bar\bt,\bar\ga,\bar\dt\in X$, there is a well-defined linear map
$$\Delta_{\bar\al,\bar\bt,\bar\ga,\bar\dt}:{}_{\bar\al+\bar\ga}
\bfU(\afsl)_{\bar\bt+\bar\dt}\ra
{}_{\bar\al}
\bfU(\afsl)_{\bar\bt}\ot
{}_{\bar\ga}
\bfU(\afsl)_{\bar\dt}
$$
such that
$\Delta_{\bar\al,\bar\bt,\bar\ga,\bar\dt}(\pi_{\bar\al+\bar\ga,\bar\bt+\bar\dt}(x))
=(\pi_{\bar\al,\bar\bt}\ot\pi_{\bar\ga,\bar\dt})(\Delta(x))$
for $x\in\bfU(\afsl)$. This collection of maps is called the comultiplication of  $\dbfU$ (see \cite{Lubk}).

Let $\dbfBn$ be the canonical basis of $\dbfU$ defined by Lusztig \cite{Lubk}.
Let $\sY(n)=\{A\in\afThnap,\,A-E\not\in\afThn\}$.
By the proof of \cite[4.3]{Lu00}, we see that for $A\in\sY(n)$, there exists a unique $\fb_A\in\dbfBn$ such that
$\dzr(\fb_A)=\th_{A+mE,r}$ if $r=\sg(A)+mn$ for some $m\geq 0$, and
$\dzr(\fb_A)=0$ otherwise. Furthermore we have
$\dbfBn=\{\fb_A\mid A\in\sY(n)\}$ (cf.  \cite{SV,Mcgerty}).

For $\la,\mu\in\afLanr$ let ${}_\la\afThnr_\mu=\{A\in\afThnr\mid
\ro(A)=\la,\,\co(A)=\mu\}$.
For $\bar\la,\bar\mu\in X$ let ${}_{\bar\la}\sY(n)_{\bar\mu}=\{
A\in\sY(n)\mid \ol{\ro(A)}=\bar\la,\,\ol{\co(A)}=\bar\mu\}$.
Then for $A\in{}_{\bar\la}\sY(n)_{\bar\mu}$ we have $\fb_A\in{}_{\bar\la}\bfU(\afsl)_{\bar\mu}$.
For $\bar\al,\bar\bt,\bar\ga,\bar\dt\in X$ and $A\in{}_{\bar\al+\bar\ga}\sY(n)_{\bar\bt+\bar\dt}$,
we write
\begin{equation}\label{hABC}
\Delta_{\bar\al,\bar\bt,\bar\ga,\bar\dt}(\fb_A)=
\sum_{B\in {}_{\bar\al}\sY(n)_{\bar\bt}\atop
C\in {}_{\bar\ga}\sY(n)_{\bar\dt}}\sfh_{A,B,C}\fb_B\ot\fb_C
\end{equation}
where $\sfh_{A,B,C}\in\sZ$.

Finally, we give a precise relation between the structure constants
of the comultiplication with respect to the canonical basis $\dbfBn$
of $\dbfU$ and that with respect to the canonical basis
$\bfBNap$ of $\afbfUslNp$.

\begin{Thm}\label{relation dbfBn bfBNap}
Let $\bar\al,\bar\bt,\bar\ga,\bar\dt\in X$, $A\in{}_{\bar\al+\bar\ga}\sY(n)_{\bar\bt+\bar\dt}$.
$B\in {}_{\bar\al}\sY(n)_{\bar\bt}$ and $C\in {}_{\bar\ga}\sY(n)_{\bar\dt}$.
Let $r'=\sg(B)$, $r''=\sg(C)$.
Assume $N>n$ and $\sfh_{A,B,C}\not=0$. Then there exist $m\in\mbn$ and
$k_0\in\mbz$ such that
$\sg(B)+\sg(C)=\sg(A)+mn$, $\ti{A_k},\ti{B_k},\ti{C_k}\in \afThNpap$ and
$$\sfh_{A,B,C}=\g_{A+mE,B,C}^{r',r''}=\up^{\lan\bfd(\ti{B_k}),
\co(\ti{C_k})\ran}\sff_{\ti{A_k},\ti{B_k},\ti{C_k}}\in\mbn[\up,\up^{-1}],$$
for $k\leq k_0$,
 where $A_k=\etak(A+mE)$, $B_k=\etak(B)$, $C_k=\etak(C)$,
$\sff_{\ti{A_k},\ti{B_k},\ti{C_k}}$ is as given in \eqref{fABC} and
$\g_{A+mE,B,C}^{r',r''}$ is as given in \eqref{gABCr}.
\end{Thm}
\begin{proof}
Note that we have $\dzra(\fb_B)=\th_{B,r'}$ and $\dzrad(\fb_C)=\th_{C,r''}$.
Let $r=r'+r''$.
Since $\dzr(\fb_A)\not=0$, we conclude that
$r=\sg(A)+mn$ for some $m\in\mbn$ and $\dzr(\fb_A)=\th_{A+mE,r}$.
Clearly we have $(\dzra\ot\dzrad)\circ\Delta_{\bar\al,\bar\bt,\bar\ga,\bar\dt}
(\fb_A)=
([\diag(\ro(B))]\ot[\diag(\ro(C)) ])\cdot\Delta_{r',r''}
(\dzr(\fb_A))\cdot([\diag(\co(B))]\ot[\diag(\co(C))])$.
This together with \ref{gABCr} and \ref{hABC} implies that
\begin{equation*}
\begin{split}
&\sum_{A'\in {}_{\bar\al}\sY(n)_{\bar\bt}\atop
A''\in {}_{\bar\ga}\sY(n)_{\bar\dt}}\sfh_{A,A',A''}
\dzra(\fb_{A'})\ot\dzrad(\fb_{A''})
=\sum_{A'\in{}_{\ro(B)}\afThnra_{\co(B)}\atop
A''\in{}_{\ro(C)}\afThnrad_{\co(C)}}\g_{A+mE,A',A''}^{r',r''}
\th_{A',r'}\ot\th_{A'',r''}.
\end{split}
\end{equation*}
It follows that
$\sfh_{A,B,C}=\g_{A+mE,B,C}^{r',r''}$. Now the result follows from
\ref{prop of canonical basis for affine q-Schur algebras}(2).
\end{proof}

\end{document}